\documentclass{amsart}

\usepackage{xcolor}

\usepackage{tikz}
\newtheorem{theorem}{Theorem}[section]
\newtheorem{lemma}[theorem]{Lemma}
\newtheorem{proposition}[theorem]{Proposition}
\newtheorem{corollary}[theorem]{Corollary}

\theoremstyle{definition}
\newtheorem{definition}[theorem]{Definition}
\newtheorem{example}[theorem]{Example}

\usepackage{amscd,amssymb}
\usepackage{amsaddr}
 \title[$S$-$\mathcal{J}$-Ideals]{$S$-$\mathcal{J}$-Ideals: A Study in Commutative and Noncommutative Rings } 
{

\author[1]{Alaa Abouhalaka $^{1}$, Hat\.{I}ce \c{C}ay $^{2,3,\ast}$ and Bayram Al\.{I} Ersoy $^{3}$}

\address[1]{$^{1}$ \quad Department of Mathematics,
\c{C}ukurova University, 01330 Balcal\i,
 Adana, Turkey; alaa1aclids@gmail.com }
\address[2]{ $^{2}$ \quad IMU Vocational School,
Istanbul Medipol University, 34810 Beykoz,
 Istanbul, Turkey.}
\address[3]{$^{3}$\quad  Department of Mathematics,
Yildiz Technical University, 34220 Esenler,
 Istanbul, Turkey; ersoya@gmail.com}

\address[*]{Corresponding author: \texttt{hcay@medipol.edu.tr}}

\begin{document}

\begin{abstract} 
In this paper, we introduce the concept of $S$-$\mathcal{J}$-ideals in both commutative and noncommutative rings. 
For a commutative ring $R$ and a multiplicatively closed subset $S$, we show that many properties of $\mathcal{J}$-ideals apply to $S$-$\mathcal{J}$-ideals and examine their characteristics in various ring constructions, such as homomorphic image rings, quotient rings, cartesian product rings, polynomial rings, power series rings, idealization rings, and amalgamation rings. In noncommutative rings, where $S$ is an $m$-system, we define right $S$-$\mathcal{J}$-ideals. We demonstrate the equivalence of $S$-$\mathcal{J}$-ideals and right $S$-$\mathcal{J}$-ideals in commutative rings with identity and provide examples to illustrate the connections between right $S$-prime ideals and $\mathcal{J}$-ideals.

\end{abstract}


\maketitle

\section{Introduction}

"A preprint has previously been published \cite{prep}".

Let $R$ be a commutative ring, and $S$ be a multiplicatively closed subset of $R$. The notion of $S$-prime ideals is first introduced in 2020 in \cite{HMS}. An ideal $I$, disjoint from $S$, of commutative ring $R$ is deemed $S$-prime if for all $a, b\in R$ with $ab\in I$, it follows that $as\in I$ or $bs\in I$. This idea has been expanded upon in many studies, including \cite{vis}, \cite{see}-\cite{mahdo}. The concept of $\mathcal{J}$-ideals first proposed in \cite{Hani1} serves as a broader generalization  of $n$-ideals   introduced in \cite{ref2}. An ideal $I$ of $R$ is termed  a $\mathcal{J}$-ideal (an $n$-ideal) if for all $a, b\in R$ with $ab\in I$, and $a\not\in \mathcal{J}(R)$ ($a\not\in\beta (R)$), then $b\in I$, where $\mathcal{J}(R)$ and $\beta (R)$ are the Jacobson and prime radicals of $R$, respectively. Recently the concept of $S$-$n$-ideals has been defined in \cite{ech2}, where an ideal $I$ disjoint from $S$ is called an $S$-$n$-ideal  if for all $a, b\in R$ with $ab\in I$, and  $as\not\in\beta (R)$, then $bs\in I$. 

In noncommutative rings, the concept of the right $S$-prime ideal is introduced in \cite{ref4}, where $S$ is an $m$-system. Definition 2.1 of \cite{ref4} states that an ideal $P$ disjoint from $S$ of a noncommutative ring $R$ is called right $S$-prime, if  and only if whenever $IK\subseteq P$, then either $I\langle s\rangle \subseteq P$ or $K\langle s\rangle\subseteq P$ for all ideals $I,K$ of $R$, and for some (fixed) $s\in S$. Just as multiplicatively closed subsets are closely linked to prime ideals in commutative rings, 
$m$-systems play a similar role in noncommutative rings. Thus, employing 
$S$ as an $m$-system is essential in noncommutative rings to obtain satisfactory results.  Additionally, the  concept of $\mathcal{J}$-ideals in noncommutative case is introduced in \cite{ref1}. Definition 5.1 of \cite{ref1} states that  a proper ideal $I$ of  $R$ is a $\mathcal{J}$-ideal if whenever $a_1, a_2 \in R$
with $a_1Ra_2\subseteq  I$ and $a_1 \not\in \mathcal{J} (R)$, then $a_2 \in I$. Recall that the concept of $\mathcal{J}$-ideals is also generalized in \cite{AA2} based on the earlier work presented in \cite{AA1}.

In this paper, we introduce the notion of $S$-$\mathcal{J}$-ideals in both commutative and noncommutative rings. The second section explores the characteristics of $S$-$\mathcal{J}$-ideals in commutative rings, demonstrating how many properties of $\mathcal{J}$-ideals naturally extend within this framework. We also examine $S$-$\mathcal{J}$-ideals in various related rings, such as homomorphic image rings, quotient rings, cartesian  product rings, polynomial rings, power series rings, idealization rings, and amalgamation rings.

The third section is dedicated to right $S$-$\mathcal{J}$-ideals in noncommutative rings, where $S$ is an $m$-system. We illustrate the equivalence of the two new concepts when $R$ is commutative with a unit, provide many examples, and clarify the relationships between right $S$-prime and $\mathcal{J}$-ideals.

\section{$S$-$\mathcal{J}$-ideals in commutative rings}

In this section, by $R$ we mean a commutative ring with identity, and $S$ will denote  a multiplicatively closed subset of $R$.

\subsection{Characteristics of $S$-$\mathcal{J}$-ideals }

\begin{definition}\label{use}  An ideal $I$ of $R$ disjoint from $S$ is called an $S$-$\mathcal{J}$-ideal, if there exists an (a fixed) $s \in S$ such that for all $a,b \in R$,  if $ab \in I$, then either $sa \in \mathcal{J}(R)$, or $sb \in I$. 

 Due to  the symmetry between $a$ and $b$, we can see  that if $I$ is an $S$-$\mathcal{J}$-ideal with $as, bs\not\in I$, then $as, bs\in\mathcal{J}(R)$.  
\end{definition}

From the definitions, it is evident that all the  $\mathcal{J}$-ideals qualify as $S$-$\mathcal{J}$-ideals.  However, the converse does not hold, as demonstrated by the upcoming example.

\begin{example}\label{Ex.}
Let $R=\mathbb{Z}_{36}$. The Jacobson radical of $R$ is $\mathcal{J} (R)=\langle \overline{6}\rangle$. The ideal $I=\langle \overline{4}\rangle$ is not a $\mathcal{J}$-ideal, since $\overline{2}.\overline{2}\in I$ but neither $\overline{2}\in \mathcal{J} (R)$ nor $\overline{2}\in I$. Take $S=\{\overline{1}, \overline{3}, \overline{9}, \overline{27}\}$. It is obvious that $S$ is a multiplicatively closed subset of $R$ and $S\cap I=\emptyset$. Now choose $s=\overline{3}\in S$ and assume $ab\in I$ for $a,b \in R$.   

 \noindent $\bullet$    If $3b\in I$, then $I$ is $S$-$\mathcal{J}$. 

\noindent $\bullet$  If $3b\notin I$, then,  $ab\in\langle \overline{4}\rangle\subseteq \langle \overline{2}\rangle$ implies $a\in \langle \overline{2}\rangle$ or $b\in \langle \overline{2}\rangle$. Assume $a\notin \langle \overline{2}\rangle$ and $b\in \langle \overline{2}\rangle$. Then,  $b=\overline{4k_1-2}$, for $k_1\in\mathbb{Z}^+$, and since $a\notin \langle \overline{2}\rangle$, then $a=\overline{2k_2-1}$, for $k_2\in\mathbb{Z}^+$. Notice that $ab= (\overline{2k_2-1})(\overline{4k_1-2})\notin I$, a contradiction. Hence $a$ must be even, and thus $3a\in \mathcal{J} (R)$. Therefore, $I$ is  $S$-$\mathcal{J}$. 
\end{example}

\begin{proposition}\label{1}  If $I$ is an $S$-$\mathcal{J}$-ideal of $R$, associated with an element $s\in S$, then $I \subseteq (\mathcal{J}(R):s)$. 
\end{proposition}
\begin{proof}
 Assume $I$ is an $S$-$\mathcal{J}$-ideal. If $I\not\subseteq (\mathcal{J}(R):s)$. Then there exists $a \in I$ such that $as\not\in \mathcal{J}(R)$. However $a·1 \in I$, so  $s\in I$, which contradicts with $S\cap I=\emptyset$.  Hence, $I$ is contained in $(\mathcal{J}(R):s)$. 
\end{proof}

\begin{corollary}\label{WH} \protect[Proposition 2.2 of \cite{Hani1}]  If $I$ is a $\mathcal{J}$-ideal of $R$, then $I \subseteq \mathcal{J}(R)$. 

\end{corollary}
\begin{proof} Follows from Proposition \ref{1} by taking $S=\{1\}$.
\end{proof}

\begin{corollary} In a ring $R$ we have the following: 
\end{corollary}
$(1)$ Every $S$-$n$-ideal is an $S$-$\mathcal{J}$-ideal. 

$(2)$ Every $n$-ideal is a $\mathcal{J}$-ideal  [Proposition 2.4 of \cite{Hani1}]. 

$(3)$ $\mathcal{J} (R)$ is an $S$-$\mathcal{J} $-ideal if and only if $\mathcal{J} (R)$
is  $S$-prime.

\begin{proof} $(1)$ Follows from the fact that the prime radical is contained in 
the Jacobson radical. 

 $(2)$ Follows from $(1)$ by taking $S=\{1\}$.
 
 $(3)$ Obvious. 
\end{proof}

Actually, in $(3)$ of the previous corollary, It was shown that the Jacobson radical 
$\mathcal{J} (R)$ is an $S$-$\mathcal{J}$-ideal if and only if $\mathcal{J} (R)$ is an $S$-prime ideal. Nevertheless, the two concepts are not comparable, as we illustrate in the following example.

\begin{example}
Let $R=\mathbb{Z}[X]$ and let $I=\langle 9x\rangle$. Considering the multiplicatively closed subset $S=\{3^n: n\in \mathbb{N}\cup \{0\}\}$. Then $\mathcal{J} (R)=\mathcal{J}(\mathbb{Z})[X]=0$.  Example 2.3 in \cite{vis}, shows that $I$ is  $S$-prime. However, since $R$ is an integral domain, then, $(\mathcal{J}(R):s)=(0:s)=0$  for  all $s\in S$, hence, $I \not\subseteq (\mathcal{J}(R):s)$. Thus, by Proposition \ref{1}, $I$ is not  $S$-$\mathcal{J}$.  
\end{example}

\begin{proposition}\label{10}  An ideal $I$ of $R$ disjoint from  $S$ is an $S$-$\mathcal{J}$-ideal, if and only if, for all ideals $A$, $B$ of $R$ with $AB \subseteq I$ there exists an (a fixed) $s \in S$ such that  either $As \subseteq \mathcal{J}(R)$, or $Bs \subseteq I$. 
\end{proposition}

\begin{proof} Suppose the ideal $I$ is $S$-$\mathcal{J}$, and let $AB \subseteq I$ for some ideals  $A$, $B$ of $R$. If $As \not\subseteq \mathcal{J}(R)$, then there exists $x\in A$ with $xs\not\in\mathcal{J}(R)$. Now for all $y\in B$, $xy\in AB\subseteq I$, hence, $ys\in I$, consequently, $Bs \subseteq I$. Conversely, let $a_1, a_2 \in R$ with $a_1a_2\in I$. Then, $\langle a_1\rangle\langle a_2\rangle\subseteq I$, and by assumption, $a_1s\in \langle a_1\rangle s\subseteq\mathcal{J}(R)$ or $a_2s\in \langle a_2\rangle s\subseteq I$. 
\end{proof}

\begin{theorem}\label{POWER}  An ideal $I$ of $R$ disjoint from $S$ is an $S$-$\mathcal{J}$-ideal, if $(I : s)$ is a  $\mathcal{J}$-ideal of $R$ for some $s \in S$. The converse holds when $\mathcal{J}(R)$ is a $\mathcal{J}$-ideal disjoint from $S$.
\end{theorem}

\begin{proof} Suppose $(I : s)$ is a  $\mathcal{J}$-ideal, and let  $a_1a_2\in I$ for some $a_1, a_2 \in R$. Then, $a_1a_2\in (I : s)$, hence, either $a_1\in \mathcal{J}(R)$ or $a_2\in (I : s)$. Consequently, either $a_1s\in \mathcal{J}(R)$ or $a_2s\in I $. Conversely, suppose the ideal $I$ is  $S$-$\mathcal{J}$, and $\mathcal{J}(R)$ is a $\mathcal{J}$-ideal with  $\mathcal{J}(R)\cap S=\emptyset$. Let  $a_1a_2\in (I : s)$ for some $a_1, a_2 \in R$, then, $a_1(a_2s)\in I$. Hence, either $a_1s\in \mathcal{J}(R)$ or $a_2s^2\in I$. If $a_1s\in \mathcal{J}(R)$, then,  $a_1\in \mathcal{J}(R)$, because $s\not\in\mathcal{J}(R)$.  If $a_2s^2\in I$, then, either $s^3\in S\cap \mathcal{J}(R)$ which is a contradiction, or  $a_2s\in I$, so $a_2\in(I:s)$. 
\end{proof}

\begin{proposition}\label{11}  An ideal $I$ of $R$ disjoint from $S$, is  $S$-$\mathcal{J}$, if and only if, for some $s\in S$, $(I:a)\subseteq (\mathcal{J}(R):s)$ for all $a\not\in (I:s)$.
\end{proposition}

\begin{proof} Suppose the ideal $I$ is $S$-$\mathcal{J}$, and let $x\in(I:a)$, then $xa\in I$, and hence, $x\in (\mathcal{J}(R):s)$.  Conversely,  let  $a_1a_2\in I$ for some $a_1, a_2 \in R$. If $a_2s\not\in I$, then, $a_2\not\in (I:s)$, hence, by assumption, $a_1\in(I:a_2)\subseteq (\mathcal{J}(R):s)$, and hence, $a_1s\in\mathcal{J}(R)$. 
\end{proof}

\begin{proposition}\label{12}  An ideal $I$ of $R$ disjoint from $S$ is  $S$-$\mathcal{J}$, if and only if, for some $s\in S$, $(I:b)\subseteq (I:s)$ for all $b\not\in (\mathcal{J}(R):s)$.
\end{proposition}

\begin{proof}  Suppose the ideal $I$ is $S$-$\mathcal{J}$, and let  $x\in(I:b)$, then, $bx\in I$, and hence, $x\in(I:s)$. Conversely,  let  $a_1a_2\in I$ for some $a_1, a_2 \in R$. If  $a_1s\not\in \mathcal{J}(R)$, then, $a_1\not\in (\mathcal{J}(R):s)$, hence, by assumption, $a_2\in(I:a_1)\subseteq (I:s)$. Thus, $a_2s\in I$.
\end{proof}

\begin{corollary} \protect[Proposition 2.10 of \cite{Hani1}] Let  $I$ be an ideal of $R$. The following are equivalent. 
\end{corollary}
(1) $I$ is a $\mathcal{J}$-ideal of $R$. 

(2) $I = (I : b)$ for all $b \not\in \mathcal{J}(R)$. 

(3) For ideals $A$ and $B$ of $R$, $AB \subseteq I$ and $A\not\subseteq \mathcal{J}(R)$ we have  $B \subseteq I$.

 (4) $(I : a) \subseteq \mathcal{J}(R)$ for all $a \not\in I$.

\begin{proof}   Consider $S=\{1\}$ in Proposition  \ref{11} and Proposition \ref{12} . Then, $(I:s)=I$ and $(\mathcal{J}(R):s)=\mathcal{J}(R)$. Hence, 

$(1)\Rightarrow (2)$: By  Proposition \ref{12}, $(I:b)\subseteq (I:s)=I\subseteq(I:b)$.

$(2)\Rightarrow (3)$: By Proposition \ref{12} and Proposition \ref{10}. 

$(3)\Rightarrow (1)$: By Proposition \ref{10}.

$(3)\Rightarrow (4)$ and $(4)\Rightarrow (3)$: Follows from Proposition \ref{11}. 
\end{proof}

If $I$ is an ideal within  $R$, then,  considering $I$ its own ring, Example   4.7 of \cite{ref5} shows that  the intersection of $\mathcal{J}(R)$ and  $I$ is precisely  $\mathcal{J}(I)$,  the Jacobson radical of the ring $I$.

\begin{corollary} Let $I$ be  an $S$-$\mathcal{J}$-ideal of $R$, and let $P$ be an ideal of the ring $I$ such that $P=(P:i)$ for all $i\in I$. Then, $P$ is an $S$-$\mathcal{J}$-ideal of the ring $I$.
\end{corollary}

\begin{proof}  Let $a_1a_2\in P$ for some $a_1, a_2\in I$. Then, $a_1a_2\in I$, and hence, $a_1s\in\mathcal{J}(R)$ or $a_2s\in I$. If  $a_1s\in\mathcal{J}(R)$, then, $a_1s\in\mathcal{J}(R)\cap I=\mathcal{J}(I)$. If $a_2s\in I$, then, $a_1a_2s\in I\cap P=P$, and hence, $a_2s\in (P:a_1)=P$. 
\end{proof}

An ideal $P$ of a ring $R$ is termed Jacobson if it can be represented as the intersection of  maximal ideals of $R$. While an ideal $K$ is called $S$-finite (in the sense of \cite{DD}) if there exists a finitely generated ideal $F$ such that $Ks \subseteq F \subseteq K$ for some $s\in S$. 
\begin{corollary} Let $I$ be an intersection of maximal ideals of $R$. If $I$ is  a finitely generated $S$-$\mathcal{J}$-ideal, then, $\mathcal{J}(R)$ is $S$-finite. Particularly, if a Jacobson ideal is finitely generated, then, $\mathcal{J}(R)$ is $S$-finite.
\end{corollary}

\begin{proof}  Since $I$ is an intersection of maximal ideals, then, $\mathcal{J}(R)\subseteq I$,  and since $I$ is $S$-$\mathcal{J}$, then by Proposition \ref{1}, $\mathcal{J}(R)\subseteq I\subseteq(\mathcal{J}(R):s)$ for some $s\in S$. Thus, $\mathcal{J}(R)s\subseteq Is\subseteq\mathcal{J}(R)$, and since $Is$ is finitely generated, then,  $\mathcal{J}(R)$ is $S$-finite.
\end{proof}

Similar to the proof of the above corollary, one can show that for a maximal finitely generated ideal $M$, if $M$ is $S$-$\mathcal{J}$, then, $\mathcal{J}(R)$ is $S$-finite.

\begin{proposition}\label{HHHH} Let $I$, $J$, and $A$ be  ideals of $R$ such that  $A\not\subseteq (\mathcal{J}(R):s)$ for all $s\in S$. Then, we have: 
\end{proposition}

$(1)$ If $I$ and $J$ are   $S$-$\mathcal{J}$-ideals, $I$ ($J$) is finitely generated, and $AI=AJ$, then, $J$ ($I$)  is $S$-finite.

$(2)$ If $AI$ is  finitely generated $S$-$\mathcal{J}$-ideal, then, $I$ is $S$-finite.

\begin{proof} $(1)$ Suppose  $I$ is a finitely generated.  Since $AI\subseteq J$, then, $Is_1\subseteq J$ for some $s_1\in S$. Similarly there exists $s_2 \in S$ such that $Js_2\subseteq I$. Hence, $Js_2s_1\subseteq Is_1\subseteq J$. For $s=s_1s_2\in S$,  $Js\subseteq K \subseteq J$, where $K =Is_1$ is finitely generated. Thus $J$ is $S$-finite.

$(2)$ Since $AI\subseteq AI$, then, $Is\subseteq AI\subseteq I$. 
\end{proof}

\begin{corollary} \protect[Proposition 2.21 of \cite{Hani1}]  Let $I$, $J$, and $A$ be  ideals of $R$ such that  $A\not\subseteq \mathcal{J}(R)$. Then the following hold. 
\end{corollary}
$(1)$ If $I$ and $J$ are $\mathcal{J}$-ideals of $R$ and $AI = AJ$, then $I = J$. 

$(2)$ If $I$ is an ideal such that $IA$ is an  $\mathcal{J}$-ideal, then $AI = I$.
\begin{proof} Follows from Proposition \ref{HHHH} by taking $S=\{1\}.$
\end{proof}

\begin{lemma}\label{13} Let  $\emptyset\neq X$ be a subset of R. If $I$ is an $S$-$\mathcal{J}$-ideal of $R$ and  $X\not\subseteq I$, then $(I : X)$ is  an $S$-$\mathcal{J}$-ideal of $R$. 
\end{lemma}

\begin{proof}   Let $A_1A_2 \subseteq (I : X)$ for some ideals $A_1, A_2$ of $R$. Then, $A_1(A_2 X)\subseteq I  $, (notice that $A_2 X$ is an ideal). If $A_1s\not\subseteq\mathcal{J}(R)$ for all $s\in S$, then, $(A_2 X)s\subseteq I$, consequently,  $A_2s \subseteq (I:X)$.
\end{proof}

\begin{proposition}\label{14} In a ring $R$ we have:
\end{proposition}
 $(1)$ An ideal $I$ that is maximal with respect to being an $S$-$\mathcal{J}$-ideal is prime. 

$(2)$ If a prime ideal $I$ (disjoint from $S$) have the property $I=(\mathcal{J}(R):s)$ for some $s\in S$, then, $I$ is maximal with respect to being an $S$-$\mathcal{J}$-ideal.

\begin{proof}  
 $(1)$ Let $a_1a_2\in I$ for some $a_1, a_2\in R$. If $a_1\not\in I$, then, by Lemma \ref{13}, $(I:a_1)$ is $S$-$\mathcal{J}$ containing $I$, and by the maximality of $I$ , $(I:a_1)=I$, hence, $a_2\in I$.

 $(2)$ Let $a_1a_2\in I$ for some $a_1, a_2\in R$. Since $I$ is prime then, either $a_1\in I=(\mathcal{J}(R):s)$, and hence, $a_1s\in \mathcal{J}(R)$, or $a_2\in I$, and thus $a_2s\in I$. Now let $J$ be any $S$-$\mathcal{J}$-ideal. Then, by Proposition \ref{1}, $J\subseteq (\mathcal{J}(R):s)=I$, which proves the maximality of $I.$
\end{proof}

It is worth noting that in $(2)$ of the above proposition, the inclusion $I\subseteq(\mathcal{J}(R):s)$ was sufficient to show that  the prime ideal $I$ is $S$-$\mathcal{J}$. Consequently, the converse of Proposition \ref{1} is valid with the added condition that   $I$ is prime. We will now present another case  where the converse of Proposition \ref{1} also applies. Recall that $\mathcal{J}^*(I)$ represents  the Jacobson radical of the ideal $I$, defined as the intersection of all maximal ideals  containing $I$. It is clear that $\mathcal{J}(R)\subseteq\mathcal{J}^*(I)$.

\begin{proposition}\label{RR} Let $I$ be an ideal of  $R$ disjoint from $S$, and let $(\mathcal{J}(R): s)=\mathcal{J}(R)$ for some $s\in S$.  Then, $I$ is  $S$-$\mathcal{J}$  associated with $s\in S$, if and only if for all $a_1, a_2 \in R$ with $a_1a_2 \in I$, either $a_1s\in\mathcal{J}^*(I)$ or $ a_2s\in I$, and  $I\subseteq(\mathcal{J}(R):s)$.
\end{proposition}

\begin{proof}  Suppose $I$ is  an $S$-$\mathcal{J}$-ideal  associated with $s\in S$. Then, by  Proposition \ref{1}, $I\subseteq(\mathcal{J}(R):s)$. Now let $a_1a_2 \in I$ for some $a_1, a_2 \in R$, then by assumption, either $a_1\in (\mathcal{J}(R):s)\subseteq(\mathcal{J}^*(I):s)$ or $ a_2\in (I: s)$ for some $s\in S$. Conversely, suppose  $a_1a_2 \in I$  for some $a_1, a_2 \in R$. Then, $a_1s\in\mathcal{J}^*(I)$ or $ a_2s\in I$.   Now since $I\subseteq (\mathcal{J}(R): s)=\mathcal{J}(R)$, then, $\mathcal{J}^*(I)\subseteq\mathcal{J}(\mathcal{J}(R))=\mathcal{J}(R)$. Thus, either $a_1s\in  \mathcal{J}(R)$ or $ a_2s\in I$. Hence, $I$ is an $S$-$\mathcal{J}$-ideal of $R$. 
\end{proof}
\begin{corollary} \protect [Proposition 2.13 of \cite{Hani1}] An ideal $I$ that is maximal with respect to being a $\mathcal{J}$-ideal is prime.  If, in particular, $I = \mathcal{J}(R)$, then the converse is true.
\end{corollary} 
\begin{proof} Follows by Proposition \ref{14} by taking $S=\{1\}$.
\end{proof}

\subsection{$S$-$\mathcal{J}$-ideals in related rings}

 Here, we explore the connections between $S$-$\mathcal{J}$-ideals of a ring $R$ and those in various related rings, such as homomorphic image rings, quotient rings, cartesian product rings, polynomial rings, power series rings, idealization rings, and amalgamation rings.
\begin{theorem}\label{teo}
Let $\psi :R_1 \rightarrow R_2$ be a ring epimorphism and $S$ be a multiplicatively closed subset of $R_1$. Then the following hold.
\end{theorem}
$(1)$ If $P$ is an $S$-$\mathcal{J}$-ideal of $R_1$ with $\text{Ker}(\psi)\subseteq P$, then $\psi (P)$ is an $\psi (S)$-$\mathcal{J}$-ideal of $R_2$.

$(2)$ If $L$ is an $\psi (S)$-$\mathcal{J}$-ideal of $R_2$ with $\text{Ker}(\psi)\subseteq \mathcal{J}(R_1)$, then $\psi^{-1} (L)$ is an $S$-$\mathcal{J}$-ideal of $R_1$.

\begin{proof}
First, we need to show that $\psi (P)\cap \psi (S)=\emptyset$. Otherwise, there exists $x\in \psi (P)\cap \psi (S)$ that implies $x=\psi (a)=\psi (s)$ for some $a\in P$ and $s\in S$. Hence $a-s\in \text{Ker}(\psi)\subseteq P$ and $s\in P$, a contradiction.

$(1)$ Let $u,v\in R_2$ and $uv\in \psi (P)$. Since $\psi$ is an epimorphism $u=\psi (a)$ and $v=\psi (b)$ for some $a,b\in R_1$. Since $\psi (a)\psi (b)\in \psi (P)$ and $\text{Ker}(\psi)\subseteq P$, we get $ab\in P$ and thus there is an $s\in S$ such that $sa\in \mathcal{J} (R_1)$ or $sb\in P$. Hence $\psi (s)u\in\psi(\mathcal{J} (R_1)) \subseteq\mathcal{J} (R_2)$ or $\psi (s)v\in \psi (P)$.

\medskip
$(2)$ Let $x,y\in R_1$ with $xy\in \psi^{-1} (L)$, and assume $xs\not\in \mathcal{J} (R_1)$ for all $s\in S$.   Then $\psi (xy)=\psi (x)\psi (y)\in L$ and since $L$ is an $\psi (S)$-$\mathcal{J}$-ideal of $R_2$, there exists $\psi (s)\in \psi (S)$ such that $\psi (s)\psi (x)\in \mathcal{J} (R_2)$ or $\psi (s)\psi (y)\in L$. If $\psi (s)\psi (x)=\psi(xs)\in \mathcal{J} (R_2)$,  then, $\psi(xs)$ is contained in each maximal ideal of $R_2$, and for any maximal ideal $\mathcal{M}_1$ of $R_1$, $\psi(\mathcal{M}_1)$ is also maximal of $R_2$, and   $\text{Ker}(\psi)\subseteq\mathcal{J} (R_1)\subseteq\mathcal{M}_1$, hence, $xs\in\psi^{-1}(\psi(xs))\subseteq\psi^{-1}(\psi(\mathcal{M}_1))=\mathcal{M}_1$, consequently, $xs\in \mathcal{J} (R_1)$, a contradiction. Thus, $\psi (ys)\in L$ and hence, $ys\in \psi^{-1} (L)$.
\end{proof}

\begin{proposition}
Let $P_1,P_2\lhd R$ with $P_1\subseteq P_2$, and  $\overline S=\{s+P_1: s\in S\}$. Then the following hold.
\end{proposition}
$(1)$ If $P_2$ is an $S$-$\mathcal{J}$-ideal of $R_1$, then $P_2/ P_1$ is an $\overline S$-$\mathcal{J}$-ideal of $R/ P_1$.

\medskip
$(2)$ If $P_2/ P_1$ is an $ \overline S$-$\mathcal{J}$-ideal of $R/ P_1$ and $P_1\subseteq \mathcal{J} (R)$, then $P_2$ is an $S$-$\mathcal{J}$-ideal of $R$. 

\medskip
$(3)$ If $P_2/P_1$ is an  $ \overline S$-$\mathcal{J}$-ideal of $R/P_1$ and $P_1$ is a  $\mathcal{J}$-ideal of $R$, then $P_2$ is an $S$-$\mathcal{J}$-ideal of $R$.

\begin{proof}
$(1)$ Suppose $P_2$ is an $S$-$\mathcal{J}$-ideal. Let $\psi : R\rightarrow R/ P_1$ be the natural epimorphism defined by $\psi (r)=r+P_1$, for $r\in R$. Recall that $\text{Ker}(\psi)=P_1\subseteq P_2$. Hence by Theorem \ref{teo} $(1)$, we get $\psi (P_2)=P_2/ P_1$ is an $\overline S$-$\mathcal{J}$-ideal of $R/ P_1$.

\medskip
$(2)$ By considering the natural epimorphism $\psi :R\rightarrow R/ P_1$, since $P_1\subseteq \mathcal{J} (R)$, then by using Theorem \ref{teo} $(2)$, $P_2=\psi^{-1} (P_2/ P_1)$ is an $S$-$\mathcal{J}$-ideal of $R$.

\medskip
$(3)$ By $(2)$ and Corollary \ref{WH}. 
\end{proof}

\begin{corollary}
Let $\{P_i :i\in \Delta\}$ be a nonempty family of $S$-$\mathcal{J}$-ideal of $R$. Then $\bigcap_{i\in \Delta} P_i$ is an $S$-$\mathcal{J}$-ideal of $R$. 
\end{corollary}

\begin{proof} It is clear that $S\cap(\bigcap_{i\in \Delta} P_i)=\emptyset.$
Let $x,y\in R$ such that $yx\in \bigcap_{i\in \Delta} P_i$ and $xs\notin \mathcal{J} (R)$ for some $s\in S$. Then $xy\in P_i$ for each $i\in \Delta$. Since $P_i$ is an $S$-$\mathcal{J}$-ideal for all $i\in \Delta$, we have $ys\in P_i$ for each $i\in \Delta$ and so $ys \in \bigcap_{i\in \Delta} P_i$. Hence $\bigcap_{i\in \Delta} P_i$ is an $S$-$\mathcal{J}$-ideal of $R$. 
\end{proof}

 Referring to \cite{Hani1}, Proposition 2.30 states that $R_1\times R_2$ has no $\mathcal{J}$-ideal for any rings $R_1$ and $R_2$. However, the following theorem shows that $R_1\times R_2$ can have  an $S$-$\mathcal{J}$-ideal if at least one of the rings $R_1$ or $R_2$ possesses such an ideal.

\begin{theorem}\label{MO10} Let $I_1$, $I_2$ be ideals of $R_1$, $R_2$, respectively. For any multiplicatively closed subsets $S_1$ and $S_2$ of $R_1$ and $R_2$, respectively, we have the following: 
\end{theorem}
$(1)$ $I_1\times R_2$ is an $(S_1\times S_2)$-$\mathcal{J}$-ideal of $R_1\times R_2$, if and only if,  $I_1$ is an $S_1$-$\mathcal{J}$-ideal and $\mathcal{J}(R_2)\cap S_2\neq \emptyset$.

\medskip

$(2)$ $R_1\times I_2$ is an $(S_1\times S_2)$-$\mathcal{J}$-ideal of $R_1\times R_2$, if and only if,  $I_2$ is an $S_2$-$\mathcal{J}$-ideal and $\mathcal{J}(R_1)\cap S_1\neq \emptyset$.

\begin{proof} $(1)$ Let  $R=R_1\times R_2$, and $S=S_1\times S_2$. It is clear that  $(I_1\times R_2)\cap S=\emptyset$, if and only if $I_1\cap S_1=\emptyset$.  Suppose  $I=I_1\times R_2$ is an $S$-$\mathcal{J}$-ideal of $R$,  since $(0, 1)(1, 0)\in I$, then, either  $ (0, 1) (s_1, s_2)\in \mathcal{J}(R_1\times R_2)=\mathcal{J}(R_1)\times\mathcal{J}(R_2)$ or $ (1, 0) (s_1, s_2)\in I$ for some $(s_1, s_2)\in S$, hence, either $  s_2\in \mathcal{J}(R_2)$, and hence, $\mathcal{J}(R_2)\cap S_2\neq \emptyset$, or $   s_1\in I_1$, which is a contradiction. Now let $a_1a_1^{\prime}\in I_1$ for some  $a_1, a_1^{\prime}\in R_1$, then, $(a_1, 1)(a_1^{\prime}, 1)\in I$. Hence, $ (a_1, 1) (s_1, s_2)\in \mathcal{J}(R_1\times R_2)=\mathcal{J}(R_1)\times\mathcal{J}(R_2)$ or $ (a_1^{\prime}, 1) (s_1, s_2)\in I$ for some $(s_1, s_2)\in S$, consequently, either  $a_1s_1\in\mathcal{J}(R_1)$, or $a_1^{\prime}s_1\in I_1$. Thus, the ideal $I_1$ is  $S_1$-$\mathcal{J}$ of $R_1$. Conversely,  suppose the ideal $I_1$ is  $S_1$-$\mathcal{J}$ and $\mathcal{J}(R_2)\cap S_2\neq \emptyset$. Let $(a_1, b_2)(a_1^{\prime}, b_2^{\prime})\in I$ for some $(a_1, b_2), (a_1^{\prime}, b_2^{\prime})$ of $R$, then, $a_1a_1^{\prime}\in I_1$, and thus, $a_1s_1\in\mathcal{J}(R_1)$, or $a_1^{\prime}s_1\in I_1$ for some $s_1 \in S_1$. For any $s_2\in \mathcal{J}(R_2)\cap S_2$, we have $ (a_1, b_2) (s_1, s_2)\in \mathcal{J}(R_1\times R_2)=\mathcal{J}(R_1)\times\mathcal{J}(R_2)$ or $ (a_1^{\prime}, b_2^{\prime}) (s_1, s_2)\in I$. Thus, $I=I_1\times R_2$ is an $S$-$\mathcal{J}$-ideal  of $R$.  

\medskip
$(2)$ Similar to $(1)$. 
\end{proof}

In  Theorem \ref{MO10} the conditions $\mathcal{J}(R_2)\cap S_2\neq \emptyset$ and $\mathcal{J}(R_1)\cap S_1\neq \emptyset$ can not be omitted. Indeed, by taking $R_1=R_2=\mathbb{Z}_{36}$, $S_1=S_2=\{\overline{1}, \overline{3}, \overline{9}, \overline{27}\}$ and $P=\langle \overline{4}\rangle$ we know from the Example \ref{Ex.} that $P$ is an $S$-$\mathcal{J}$-ideal of $R_1$ but $P\times R_1$ is not an $(S_1\times S_2)$-$\mathcal{J}$-ideal of $R_1\times R_2$ since $(\overline{2},\overline{1})(\overline{2},\overline{1})\in P\times R_1$ but for each $(s_1,s_2)\in S_1\times S_2$, neither $(s_1,s_2)(\overline{2},\overline{1})\in P\times R_1$ nor $(s_1,s_2)(\overline{2},\overline{1})\in \mathcal{J} (R)$.

\medskip
In the following, we present another example of $S$-$\mathcal{J}$-ideal which is not a $\mathcal{J}$-ideal,  by using Theorem \ref{MO10}. 

\begin{example}
Let $R=\mathbb{Z}_{36}\times\mathbb{Z}_{8}$. The Jacobson radical of $R$ is $\mathcal{J} (R)=\langle \overline{6}\rangle\times\langle \overline{2}\rangle$. Take the multiplicatively closed subset: $S_1=\{1, 3, 9, 27\}$ \text{ and let } $X=\{0, 2, 4\}$. Then, $(S_1\times X)$ is a multiplicatively closed subset (without zero) of $R$. Example \ref{Ex.} shows that the ideal $I=\langle \overline{4}\rangle$ of $\mathbb{Z}_{36}$  is  an $S_1$-$\mathcal{J}$ ideal. Thus, by Theorem \ref{MO10}, we find that the ideal $I\times\mathbb{Z}_{8}$ is an $(S_1\times X)$-$\mathcal{J}$-ideal of $R$. However, $I\times\mathbb{Z}_{8}$ is not a $\mathcal{J}$-ideal due to Proposition 2.30 of \cite{Hani1}.
\end{example}

 In referring to \cite{Sharp}, the Jacobson radical of the power series ring $R[\vert x\vert]$, where $R$ is commutative with identity, is $\mathcal{J} (R[\vert x\vert]) = \mathcal{J} (R)+xR[\vert x\vert]$. While Exercises 1 (4) of \cite{AtM}, and Theorem 3 of \cite{SAA},  show that the Jacobson radical of the polynomial ring $R[x]$ is $\mathcal{J}(R[x])= \mathcal{J}(R)[x]$.

\begin{theorem} Let $R$ be a ring in which $\mathcal{J} (R)$ is a $\mathcal{J}$-ideal. An ideal $I$  of $R$, disjoint from $S$  is  $S$-$\mathcal{J}$, if and only if $I [\vert x\vert]$ is an $S$-$\mathcal{J}$-ideal of $R[\vert x\vert]$. 
\end{theorem}

\begin{proof}  Suppose $I [\vert x\vert]$ is an $S$-$\mathcal{J}$-ideal of $R[\vert x\vert]$, and let $a_1a_2\in I$ for some $a_1, a_2\in R$. Then, $a_1a_2\in I[\vert x\vert]$. Hence, either $a_1s\in\mathcal{J} (R[\vert x\vert])$ or $a_2s\in I [\vert x\vert]$ for some $s\in S$, consequently, either $a_1s\in\mathcal{J} (R)$ or $a_2s\in I $. Thus, $I$ is  $S$-$\mathcal{J}$. Conversely, suppose the ideal $I$ is  $S$-$\mathcal{J}$, and let $f_1(x)f_2(x)\in I [\vert x\vert]$ for some $f_1(x), f_2(x)\in R[\vert x\vert]$. Since $\mathcal{J} (R)$ is a $\mathcal{J}$-ideal, then, by Theorem \ref{POWER},  $(I:s)$ is a $\mathcal{J}$-ideal, and since  $f_1(x)f_2(x)\in I [\vert x\vert]\subseteq (I:s) [\vert x\vert]$, then, either $f_1(x)\in\mathcal{J} (R[\vert x\vert])$ or $f_2(x)\in(I:s) [\vert x\vert]$. If $f_1(x)\in\mathcal{J} (R[\vert x\vert])$, then, $f_1(0)\in\mathcal{J} (R)$, hence, $f_1(0)s\in\mathcal{J} (R)$, and thus, $f_1(x)s\in\mathcal{J} (R)+xR[\vert x\vert]=\mathcal{J} (R[\vert x\vert])$. If $f_2(x)\in(I:s) [\vert x\vert]=(I [\vert x\vert]:s)$, then, $f_2(x)s\in I [\vert x\vert]$. Thus, $I [\vert x\vert]$ is an $S$-$\mathcal{J}$-ideal of $R[\vert x\vert]$.  
\end{proof}

\begin{theorem} Let $R$ be a ring in which $\mathcal{J} (R)$ is a $\mathcal{J}$-ideal. An ideal $I$  of $R$, disjoint from $S$  is  $S$-$\mathcal{J}$, if and only if $I [ x]$ is an $S$-$\mathcal{J}$-ideal of $R[ x]$. 
\end{theorem}
\begin{proof} Similar to the proof of the above theorem. 
\end{proof}

In the following, we establish the relationship between $S$-$\mathcal{J}$-ideals of a ring $R$ and those of the idealization ring $R\boxplus M=\{(a, m): a\in R, m\in M\}$, where $M$ is an $R$-module, which is defined  with the usual addition and the multiplication defined as $(a_1, m_1)(a_2, m_2)=(a_1a_2, a_1m_2+a_2m_1)$, for all $(a_1, m_1), (a_2, m_2)\in R\boxplus M$. If $S$ is a multiplicatively closed subset of $R$, then $S_M=S\boxplus M$ is a multiplicatively closed subset of $R\boxplus M$. In addition, $\mathcal{J}(R\boxplus M)= \mathcal{J}(R)\boxplus M.$

\begin{theorem}\label{last}
Let  $I$ be an ideal of $R$ disjoint from $S$, and $M$ be an $R$-module. Then, $I\boxplus M$ is an $S_M$-$\mathcal{J}$-ideal of $R\boxplus M$ if and only if $I$ is an $S$-$\mathcal{J}$-ideal of $R$. 
\end{theorem}

\begin{proof} Suppose $I$ is  an $S$-$\mathcal{J}$-ideal of $R$. Notice that $S_M\cap (I\boxplus M)=\emptyset$.  Let $(a_1, m_1)(a_2, m_2)\in I\boxplus M$ for some $a_1, a_2\in R$, and $m_1, m_2\in M$. Then, $a_1 a_2\in I$, and hence, either $a_1s\in \mathcal{J}(R)$ or $a_2s \in I$ for some $s\in S$. If $a_2s \in I$, then, $(a_2, m_2)(s, 0) \in I\boxplus M$. If  $a_1s \in \mathcal{J}(R)$, then $(a_1, m_1)(s, 0)\in \mathcal{J}(R)\boxplus M=\mathcal{J}(R\boxplus M)$, where $(s, 0)\in  S_M$. Thus, $I\boxplus M$ is an $(S_M)$-$\mathcal{J}$-ideal of $R\boxplus M$.

\medskip

\noindent Conversely, suppose $I\boxplus M$ is an $(S_M)$-$\mathcal{J}$-ideal of $R\boxplus M$. Let $ab\in I$, for some $a, b\in R$. Then,  $(a, m_1)(b, m_2)\in I\boxplus M$ for some $m_1, m_2\in M$, and hence, either $(a, m_1)(s, m)\in \mathcal{J}(R\boxplus M)$, or $(b, m_2)(s, m)\in I\boxplus M$, for some $(s, m)\in S_M$. Consequently, either $as \in \mathcal{J}(R)$ or $bs \in I$. Thus, $I$ is an  $S$-$\mathcal{J}$-ideal of $R$. 
\end{proof}

\medskip
If $I$ is an ideal of $R$, then for    a submodule $N$ of $M$, $I\boxplus N$  constitutes an ideal of $R\boxplus M$ if and only if $IM \subseteq N$.

\begin{proposition} \label{prop}
Let $N$ be a submodule of an $R$-module $M$,  and $I$ be an ideal of $R$ where $IM\subseteq N$. If $I\boxplus N$ is an $S\boxplus M$-$\mathcal{J}$-ideal of $R\boxplus M$, then $I$ is an $S$-$\mathcal{J}$-ideal of $R$.
\end{proposition}

\begin{proof}
Since $S_M\cap(I\boxplus N)=\emptyset$, then,   $S\cap I=\emptyset$. Let $xy\in I$ for some $x,y\in R$. Then $(x,0)(y,0)\in I\boxplus N$ thus $(x,0)(s,n)\in \mathcal{J} (R)\boxplus M$ or $(y,0)(s,n)\in I\boxplus N$ for some $(s,n)\in S_M$. Thus $xs\in \mathcal{J} (R)$ or $ys\in I$ and thus $I$ is an $S$-$\mathcal{J}$-ideal of $R$.
\end{proof}

The converse of Proposition \ref{prop} is not true in general as we show next.

\begin{example} In $R=\mathbb{Z}$, the zero ideal is prime and $\mathcal{J}(\mathbb{Z})=0$, hence, $(\mathcal{J}(\mathbb{Z}):s)=(0:s)=0$,  for any $s\in S=\{1,3, 9, 27, \cdots\}$, because $\mathbb{Z}$ is an integral domain. Thus, the paragraph above the Proposition \ref{RR}, the zero ideal is $S$-$\mathcal{J}$. Take the $\mathbb{Z}$-module $\mathbb{Z}_{6}$, then,  $0\boxplus 0$ is not an $(S\boxplus \mathbb{Z}_{6})$-$\mathcal{J}$-ideal, because $(2,0)(0, 3) \in 0\boxplus0$, however, $(2s,2m)=(s,m)(2,0)\not\in \mathcal{J}(\mathbb{Z})\boxplus Z_6 = 0\boxplus Z_6$ and $(0, 3s)=(s,m)(0,3) \not\in 0\boxplus 0$ for all $(s,m) \in S\boxplus Z_6$.
\end{example}

Let $R$ and $A$ be two commutative rings, $f: R \to A$ be a ring homomorphism, and $J$ be an ideal of the ring $A$.
The amalgamation of $R$ with $A$ along $J$ with respect to $f$, denoted by $R \bowtie^f J$, is the subring of $R\times A$ given by 
\[
R\bowtie^f J=\{(r,f(r)+j) \vert r\in R, j\in J\}. 
\]
This construction is a generalization of the amalgamated duplication of a ring  along an ideal. Recall that  for an ideal $J$ of a ring $R$, the amalgamated duplication of $R$ along $J$, denoted by $R\bowtie J$, is the subring of $R\times R$ given by 
\[
R \bowtie J=\{(r,r+j) \vert r\in R, j\in J\}.
\]
In addition, if $I$ is an ideal of $R$, then, $I \bowtie^f J=\{(i,f(i)+j) \vert i\in I, j\in J\}$, is an ideal of $R \bowtie^f J$. Moreover, for a multiplicatively closed subset $S$ of $R$,  $S^{\bowtie}=\{(s,f(s)) \vert s\in S\}$ is a multiplicatively closed subset of $R \bowtie^f J$. In addition, by Lemma 2.15 of \cite{HACE}, we have $\mathcal{J}(R \bowtie^f J) = \{(r,f(r)+j): r \in \mathcal{J}(R), f(r)+j \in \mathcal{J}(f(R)+J)\}$. While by Lemma 3 of \cite {FINO}, if $J \subseteq \mathcal{J}(A)$, then  $\mathcal{J}(R \bowtie^f J) = \mathcal{J}(R) \bowtie^f J$. 

\begin{proposition}\label{WHY} Let $f: R \to A$ be a ring homomorphism, $J$ be an ideal of the ring $A$,
and $R \bowtie^f J$ be the amalgamation of $R$ with $A$ along $J$ with respect to $f$. Let $I$ be an ideal of $R$ disjoint from $S$.
\end{proposition}

$(1)$ If $I \bowtie^fJ$ is an $S^{\bowtie}$-$\mathcal{J}$-ideal of $R \bowtie^f J$, then, $I$ is an $S$-$\mathcal{J}$-ideal of $R$.

 \medskip
$(2)$ If $I$ is an $S$-$\mathcal{J}$-ideal of $R$, and $J\subseteq\mathcal{J}(A),$  then, $I \bowtie^fJ$ is an $S^{\bowtie}$-$\mathcal{J}$-ideal of $R \bowtie^f J$

\begin{proof} Since $S\cap I=\emptyset$, then, $S^{\bowtie}\cap( I\bowtie^fJ)=\emptyset$.

$(1)$ Suppose that $I \bowtie^fJ$ is an  $S^{\bowtie}$-$\mathcal{J}$-ideal of $R \bowtie^f J$, and $ab\in I$ for some $a, b \in R$. Then, $(a, f(a))(b, f(b))\in I \bowtie^fJ$, and hence by assumption, either $(a, f(a))(s,f(s))\in \mathcal{J}(R \bowtie^f J)$
or $(b, f(b))(s,f(s))\in  I\bowtie^fJ$, for some $(s,f(s))\in S^{\bowtie}$. If $(a, f(a))(s,f(s))\in \mathcal{J}(R \bowtie^f J)$, then, $as\in\mathcal{J}(R)$. If $(b, f(b))(s,f(s))\in  I\bowtie^fJ$ , then, $bs\in I$.

\medskip
$(2)$ Suppose that $I$ is an $S$-$\mathcal{J}$-ideal of $R$, and that
\[
(a, f(a)+j_1)(b, f(b)+j_2)\in I \bowtie^fJ
\]
for some $a, b\in R$ and $j_1, j_2 \in J$.
Then, $ab\in I$, and hence by assumption, either $as\in \mathcal{J}(R)$ or $bs\in I$. If $bs\in I$,
then $(b, f(b)+j_2)(s,f(s))\in  I\bowtie^fJ$. If $as\in \mathcal{J}(R)$, then, $(a, f(a)+j_1)(s,f(s))\in\mathcal{J}(R) \bowtie^f J=\mathcal{J}(R \bowtie^f J)$.
\end{proof}

\section{$S$-$\mathcal{J}$-ideals in noncommutative rings}

In this section, unless  otherwise specified,  $R$ will refer to  an associative, noncommutative ring without identity. Also,  $S$ will denote an $m$-system of $R$.  We now present our definition of right  $S$-$\mathcal{J}$-ideal,  in line with Definition 2.2  of \cite{ref4}.

\begin{definition}\label{NON}
An ideal $P$ of $R$, that is disjoint from $S$, is called a right $S$-$\mathcal{J}$-ideal if and only if whenever $IK\subseteq P$, then either $I\langle s\rangle \subseteq \mathcal{J}(R)$ or $K\langle s\rangle\subseteq P$ for all ideals $I,K$ of $R$, and for some (fixed) $s\in S$.
\end{definition}

Definition 5.1 of \cite{ref1} states that an  ideal $P$ of  $R$ is a $\mathcal{J}$-ideal if whenever $a_1, a_2 \in R$ with $a_1Ra_2\in  P$, then,  $a_1\in \mathcal{J}(R)$ or  $a_2 \in P$. Equivalently, for ideals $I$, $K$ of $R$, with $IK \subseteq P$ either $I \subseteq \mathcal{J}(R)$ or $K\subseteq P$.  It is easy to check that every $\mathcal{J}$-ideal, disjoint from $S$ is a right $S$-$\mathcal{J}$-ideal, and the classes of $\mathcal{J}$-ideals and right $S$-$\mathcal{J}$-ideals coincide when $S\subseteq U(R)$.

\begin{proposition}\label{H}
Let $R$ be a ring with an identity and $P\lhd R$ with $P \cap S=\emptyset$. The following are equivalent:

$(1)$ $P$ is a right  $S$-$\mathcal{J}$-ideal.

\medskip
$(2)$ For all $x,y \in R$ with $\langle x\rangle\langle y\rangle \subseteq P$, either $\langle x\rangle\langle s\rangle\subseteq\mathcal{J}(R)$ or $\langle y\rangle\langle s\rangle\subseteq P$ for some $s\in S $.

\medskip
$(3)$ For all $x, y \in R$ with $xRy \subseteq P$, either $x\langle s\rangle\subseteq \mathcal{J}(R)$ or $y\langle s\rangle\subseteq P$ for some $s\in S $.

\end{proposition}

\begin{proof}
$(1)\Rightarrow (2)$ Clear by Definition \ref{1} .

\medskip
$(2)\Rightarrow (3)$ Let $xRy\subseteq P$ for some $x, y\in R$. Then, $\langle x\rangle\langle y\rangle \subseteq P$, hence, by $(2)$, either $x\langle s\rangle \subseteq \langle x\rangle\langle s\rangle\subseteq\mathcal{J}(R)$, or $y\langle s\rangle\subseteq \langle y\rangle\langle s\rangle\subseteq P$ for some $s\in S$.

\medskip
$(3)\Rightarrow (1)$ Suppose $IJ\subseteq P$, for some ideals $I,J$ of $R$. If $I\langle s\rangle \not\subseteq \mathcal{J}(R)$, then, there exists $a\in I$ such that $a\langle s\rangle\not\subseteq\mathcal{J} (R)$. For all $b\in J$, we have  $aRb \subseteq IJ\subseteq P$, hence, by $(3)$, we have $b\langle s\rangle\subseteq P$, and thus,  $J\langle s\rangle\subseteq P$. That completes the proof.
\end{proof}

From the above proposition, one can easily show that if $R$ is  \textit{a commutative ring with identity, and $S$ is a multiplicatively closed subset of $R$}, then, Definition \ref{use} and Definition \ref{NON} are equivalent.

\begin{proposition} Let $R$ be a commutative ring with identity, and let $S$ be a multiplicatively closed subset of $R$. Then, Definition \ref{use} and Definition \ref{NON} are equivalent. 
\end{proposition}

\begin{proof} Suppose $P$ is an  ideal of $R$ that satisfies  Definition \ref{NON}, and let $ab \in P$ for some $a, b\in R$, then,  $\langle a\rangle\langle b\rangle\subseteq P$. Since $S$ is  a multiplicatively closed subset of $R$, then $S$ is an $m$-system of  $R$. Thus, by Definition \ref{NON}, either $\langle a\rangle\langle s\rangle\subseteq P$ or $\langle b\rangle\langle s\rangle\subseteq P$ for some $s\in S$, and hence, either $as \in P$ or $bs \in P$. Now suppose $P$ satisfies  Definition \ref{use}, and let $IJ\subseteq P$ for some ideals $I,J$ of $R$.  If  $I\langle s\rangle\not\subseteq P$ and $J\langle s\rangle\not\subseteq P$ for all $s\in S$, then, there exist $i\in I$ and $j\in J$ such that  $i\langle s\rangle\not\subseteq P$ and $j\langle s\rangle\not\subseteq P$. However, $ij\in P$, thus, by Definition \ref{use}, either  $is \in P$ or $js\in P$ which implies either $i\langle s\rangle\subseteq P$ or $j\langle s\rangle\subseteq P$, contradiction. 
\end{proof}

\begin{corollary}\label{corollary} Let $P$ be an ideal of  $R$, with $P \cap S=\emptyset$. 
If $P$ is a right $S$-prime ideal such that $P\subseteq \mathcal{J}(R)$, then, $P$ is a right $S$-$\mathcal{J}$-ideal.
\end{corollary}

\begin{proof}
Let $aRb\subseteq P$ and assume $a\langle s\rangle\nsubseteq \mathcal{J} (R)$ for $a,b\in R$. Then, $a\langle s\rangle\nsubseteq P$. Since $P$ is  right $S$-prime, then $b\langle s\rangle\subseteq P$. Thus by $(3)$ of Proposition \ref{H}, $P$ is right $S$-$\mathcal{J}$.
\end{proof}

The condition $P\subseteq \mathcal{J}(R)$ in Corollary \ref{corollary} can not be omitted in general. In the following, we give  example to show this.  But first, it is good to know that  $\mathcal{J}(R[X]) = N[x]$, where $N = \mathcal{J}(R[x]) \cap R$,  see Lemma 2J of \cite{SAA}. In addition,  the prime radical of $R[x]$, $\beta(R[x])=\beta(R)[x]$, from Theorem 3 of \cite{SAA}. Moreover, in Artinian rings $\beta(R)=\mathcal{J}(R)$. 

\begin{example} \label{example}
Let $R=M_{n}\left(\mathbb{Z}_{36}\left[ x\right]\right).$  Since $\mathcal{J}(M_n(R))=M_n(\mathcal{J}(R))$, (see \cite{ref5} page 57), and $R$ is Artinian, then
{\small $$\mathcal{J}(R)=\mathcal{J}(M_{n}\left(\mathbb{Z}_{36}\left[ x\right]\right))=M_n(\mathcal{J}(\left(\mathbb{Z}_{36}\left[ x\right]\right))=M_n(\mathcal{J}\left(\mathbb{Z}_{36})\left[ x\right]\right))=M_n(\langle 6\rangle[x]))$$} 
 Consider the $m$-system: $S=\left\{ s,s^{2},s^{4},s^{8},...\right\}$, where $s=9(e_{11}+e_{22}+...+e_{nn})$, and $I=M_{n}\left( \left\langle \overline{9x}\right\rangle \right)$. It is clear that $S\cap I=\emptyset$ and $I\not\subseteq \mathcal{J}(R)$. Now let $\mathcal{N}_1R\mathcal{N}_2\subseteq I$, for some matrices $\mathcal{N}_1,\mathcal{N}_2$ in $R$. Then, there exist some matrices $\mathcal{I}_1,\mathcal{I}_2$, such that  either $\mathcal{N}_1=x\mathcal{I}_1$ or $\mathcal{N}_2=x\mathcal{I}_2$. Hence, either $\mathcal{N}_1\langle s\rangle\subseteq I$ or $\mathcal{N}_2\langle s\rangle\subseteq I$. Thus $I$ is  right $S$-prime by Proposition 2.7 of \cite{ref1}. However, by taking the ideals $A_1=M_{n}(\left\langle {3x}\right\rangle ), A_2=M_{n}(\left\langle {3}\right\rangle)$, we have  $A_1A_2\subseteq I$, however, $A_1\langle s^{\prime}\rangle\not\subseteq \mathcal{J}(R)$ and $A_2\langle s^{\prime}\rangle\not\subseteq I$, for each $s^{\prime}\in S$. Hence, $I$ is not a right $S$-$\mathcal{J}$-ideal.
\end{example}

 The following example presents an ideal which is  right $S$-$\mathcal{J}$, but  not  a $\mathcal{J}$-ideal.

\begin{example}\label{final}  
Let $R=M_{2}(\mathbb{Z}_{12})$. Let
\[
S_{R}=\Bigg\{\left[ \begin{array}{cc}s& 0\\ 0 & s\end{array} \right ]; s\in S=\{1, 3, 9\} \Bigg\}, \text{ and  } s^{\prime}=\left[ \begin{array}{cc}3& 0\\ 0 & 3\end{array} \right ].
\]
Then,  $S_{R}$ is an $m$-system of the ring $R$. Consider the ideal $P=M_{2}(\langle 4\rangle)$. Then,  $\mathcal{J}(R)=M_n(\langle 6\rangle)$, and $P\cap S_{R}=\emptyset$. Since 

\[
\left[ \begin{array}{cc}2& 0\\ 0 & 2\end{array} \right ]R\left[ \begin{array}{cc}2& 0\\ 0 & 2\end{array} \right ]\subseteq P \text{ and } \left[ \begin{array}{cc}2& 0\\ 0 & 2\end{array} \right ]\not\in\mathcal{J}(R) , \left[ \begin{array}{cc}2& 0\\ 0 & 2\end{array} \right ]\not\in P,
\]
then, $P$ is not   $\mathcal{J}$-ideal.  Now assume $\mathcal {N}_{1}R\mathcal {N}_{2}\subseteq P$, for some matrices  $\mathcal {N}_{1}, \mathcal {N}_{2}\in R$. If $\mathcal{N}_{2}\langle s^{\prime}\rangle\subseteq P$, then, $P$ is right $S_R$-$\mathcal{J}$. Thus, assume $\mathcal{N}_{2}\langle s^{\prime}\rangle\not\subseteq P$.  Then, $\mathcal {N}_{1}R\mathcal {N}_{2}\subseteq M_{2}(\langle 2\rangle)=J$, and since $J$ is prime, then either $\mathcal {N}_{1}\in J$ or  $\mathcal {N}_{2}\in J$, now we discuss two cases as the following. 

If $\mathcal {N}_{1}\in J$, then, $\mathcal {N}_{1}\langle  s^{\prime}\rangle\subseteq\mathcal{J}(R)$, and again $P$ is right $S_R$-$\mathcal{J}$. 

If $\mathcal {N}_{1}\not\in J$, then at least one of the entries of $\mathcal {N}_{1}$ is an odd number and $\mathcal {N}_{2}\in J$, hence, 
\[
\mathcal {N}_{2}=\left[ \begin{array}{cc}2k_1& 2k_2\\ 2k_3 & 2k_4\end{array} \right ], \text{ where } k_i \in \mathbb{Z}^+ \text{ for } i\in\{1, 2, 3, 4\}. 
\]
 But by the initial assumption $\mathcal{N}_{2}\langle s^{\prime}\rangle\not\subseteq P$, thus there exists $$\mathcal {I}=\left[ \begin{array}{cc}3\alpha_1& 3\alpha_2\\ 3\alpha_3 & 3\alpha_4\end{array} \right ]\in \langle s^{\prime}\rangle,$$ where $\alpha_i \in \mathbb{Z}^+$,  for  $i\in\{1, 2, 3, 4\}$,   such that $\mathcal {N}_{2}\mathcal {I}\not\in P=M_{2}(\langle 4\rangle)$, which means that at least one of the entries  of $\mathcal {N}_{2}\mathcal {I}$ is not  a multiple of $4$ in $\mathbb{Z}_{12}$. Without loss of generality, suppose that $6(k_1\alpha_1+k_2\alpha_3)$ is not a multiple of $4$, then, either $k_1$ or/and $k_2$ is/are odd. Thus, at least one of $k_i$ must be odd for $i\in\{1, 2, 3, 4\}$, consequently, at least one of the even entries of  $\mathcal {N}_{2}$ is not a multiple of $4$. By considering all the cases of the positions of the odd entries  of $\mathcal {N}_{1}$ and the not  multiple of $4$ entries of $\mathcal {N}_{2}$, we can find, in each case, an element $\mathcal {M} $ of $R$, such that $\mathcal {N}_{1}\mathcal {M}\mathcal {N}_{2}\not\in P$, i.e., $\mathcal {N}_{1}R\mathcal {N}_{2}\not\subseteq P$, which is a contradiction. That is why $\mathcal {N}_{1}$ must be an element of $J$, and thus, $\mathcal {N}_{1}\langle  s^{\prime}\rangle\subseteq\mathcal{J}(R)$. 
\end{example}

\begin{theorem}
If $R$ has an identity, and $P\lhd R$ with $P \cap S=\emptyset$, then,  for some $s\in S$,  $(P:\langle s\rangle)$ is a right $S$-$\mathcal{J}$-ideal, if and only if, $P$ is right $S$-$\mathcal{J}$.
\end{theorem}

\begin{proof} Suppose $(P:\langle s\rangle)$ is a right $S$-$\mathcal{J}$-ideal, and let  $xRy \subseteq P$ for some $x, y \in R$, then,  $xRy \subseteq (P:\langle s\rangle)$, hence by assumption, either $x\langle s_1\rangle\subseteq \mathcal{J} (R)$ or $y\langle s_1\rangle\subseteq (P:\langle s\rangle)$ for some $s_1\in S$. If $y\langle s_1\rangle\subseteq (P:\langle s\rangle)$, then,  $y\langle s_1\rangle\langle s\rangle\subseteq P$, hence there exists $s^{\prime}=s_1rs$ for some $r\in R$ such that $y\langle s^{\prime}\rangle\subseteq y\langle s_1\rangle\langle s\rangle \subseteq P$. If  $x\langle s_1\rangle \subseteq \mathcal{J} (R)$, then similarly, we find $x\langle s^{\prime}\rangle \subseteq \mathcal{J} (R)$. Thus,  $P$ is right $S$-$\mathcal{J}$. Conversely, suppose  $P$ is right $S$-$\mathcal{J}$, and let  $xRy \subseteq (P:\langle s\rangle)$ for some $x, y \in R$, then,  $\langle x\rangle\langle y\rangle\langle s\rangle \subseteq P$, hence, by assumption,  either $\langle x\rangle\langle s_1\rangle\subseteq \mathcal{J} (R)$ or $\langle y\rangle\langle s\rangle\langle s_1\rangle\subseteq P$ for some $s_1\in S$. If $\langle y\rangle\langle s\rangle\langle s_1\rangle\subseteq P$, then,   $\langle y\rangle\langle s\rangle\langle s_1\rangle\subseteq (P:\langle s\rangle)$, hence there exists $s^{\prime}=srs_1$ for some $r\in R$ such that $y\langle s^{\prime}\rangle\subseteq\langle y\rangle\langle s\rangle\langle s_1\rangle\subseteq (P:\langle s\rangle)$. If  $\langle x\rangle\langle s_1\rangle \subseteq \mathcal{J}(R)$, then similarly, we find $x\langle s^{\prime}\rangle \subseteq \mathcal{J}(R)$. Thus, $(P:\langle s\rangle)$ is right $S$-$\mathcal{J}$.
\end{proof}

\begin{proposition}
Suppose $R$ has an identity, and $P\lhd R$ such that $P \cap S=\emptyset$. If for some $s\in S$,  $(P:\langle s\rangle)$ is a $\mathcal{J}$-ideal, then, $P$ is right $S$-$\mathcal{J}$.
\end{proposition}

\begin{proof} Suppose $(P:\langle s\rangle)$ is a $\mathcal{J}$-ideal, and let  $xRy \subseteq P$ for some $x, y \in R$, then,  $xRy \subseteq (P:\langle s\rangle)$, hence by assumption, either $x \in \mathcal{J}(R)$ or $y\in (P:\langle s\rangle)$. If  $x \in \mathcal{J}(R)$, then $x\langle s\rangle \subseteq \mathcal{J}(R)$. If $y\in (P:\langle s\rangle)$, then,  $y\langle s\rangle\subseteq P$. Thus, by $(3)$ of Proposition \ref{H}, $P$ is right $S$-$\mathcal{J}$.
\end{proof}

The following proposition shows the converse of the above proposition holds when $S$ is contained in the center of the ring $R$ ($C(R)$), and  $(\mathcal{J}(R):\langle s\rangle)$ is a $\mathcal{J}$-ideal, for some $s\in S$. 

\begin{proposition}
Let $P$ be an ideal of a ring $R$ with identity,  $S$ be an $m$-system of $R$ such that $S\subseteq C(R)$, and $(\mathcal{J}(R):\langle s\rangle)$ is a $\mathcal{J}$-ideal, disjoint from $S$, for some $s\in S$. If $P$ is right $S$-$\mathcal{J}$, then, $(P:\langle s\rangle)$ is a $\mathcal{J}$-ideal.
\end{proposition}

\begin{proof}  Since $(\mathcal{J}(R):\langle s\rangle)$ is a $\mathcal{J}$-ideal, then, by Proposition 1.5 of \cite{ref1},  $(\mathcal{J}(R):\langle s\rangle)\subseteq \mathcal{J}(R)$, and hence, $(\mathcal{J}(R):\langle s\rangle)= \mathcal{J}(R)$. Let $xRy \subseteq (P:\langle s\rangle)$ for some $x, y\in R$, then, $\langle x\rangle\langle y\rangle\langle s\rangle\subseteq P$, and by assumption, either  $\langle x\rangle\langle s\rangle\subseteq\mathcal{J}(R)$ or $\langle y\rangle\langle s\rangle\langle s\rangle\subseteq P$. If   $\langle x\rangle\langle s\rangle\subseteq\mathcal{J}(R)$, then, $x\in (\mathcal{J}(R):\langle s\rangle)=\mathcal{J}(R)$. If $\langle y\rangle\langle s\rangle\langle s\rangle\subseteq P$, then, $\langle s\rangle\langle s\rangle\langle y\rangle\subseteq P$, and again either $\langle s\rangle^3\subseteq \mathcal{J}(R)$, which implies for some $r\in R$, $s^{\prime}=srs\in \langle s\rangle^2\subseteq (\mathcal{J}(R):\langle s\rangle)=\mathcal{J}(R)$, contradiction, or $\langle y\rangle\langle s\rangle\subseteq P$, which implies $y\in (P:\langle s\rangle)$. Thus, $(P:\langle s\rangle)$ is a $\mathcal{J}$-ideal.
\end{proof}

\begin{theorem} Let $f$: $R_1\to R_2$ be a ring epimorphism, and $P$ be  a right $S$-$\mathcal{J}$-ideal of $R_1$ such that
$\text{\rm Ker}(f)\subseteq P$. Then $f(P)$ is a right $f(S)$-$\mathcal{J}$-ideal of $R_2$.
\end{theorem}

\begin{proof} 
Let $xR_2y\subseteq f(P)$ and $x\langle f(s)\rangle\not\subseteq \mathcal{J} (R_2)$ for $x,y\in R_2$. Since $f$ is an epimorphism, there exists $c,d\in R_1$ such that $f(c)=x$ and $f(d)=y$. Thus, $xR_2y=f(c)f(R_1)f(d)=f(cR_1d)\subseteq f(P)$. Since $\text{\rm Ker}(f)\subseteq P$ we obtain $cR_1d\subseteq P$. Since $P$ is  $S$-$\mathcal{J}$ of $R_1$, then, either $c\langle s\rangle\subseteq \mathcal{J} (R_1)$ or $d\langle s\rangle\subseteq P$ for some $s\in S$. If $c\langle s\rangle\subseteq \mathcal{J} (R_1)$, then, since $\mathcal{J}$ is a complete, idempotent Hoehnke radical, we have  $f(\mathcal{J} (R_1))\subseteq \mathcal{J} (R_2)$. Hence, $x\langle f(s)\rangle =f(c)\langle f(s)\rangle\subseteq f(\mathcal{J} (R_1))\subseteq \mathcal{J} (R_2)$ that is a contradiction. Hence, $c\langle s\rangle\not\subseteq \mathcal{J} (R_1)$ and thus,  $d\langle s\rangle\subseteq P$. Since $y=f(d)$, then $y\langle f(s)\rangle\subseteq f(P)$, consequently,  $f(P)$ is a right $f(S)$-$\mathcal{J}$-ideal of $R_2$.
\end{proof}

\begin{theorem} 
Let $f$: $R_1\to R_2$ be a ring epimorphism, and $P$ be an ideal of $R$ such that $\text{\rm Ker}(f)\subseteq P\cap\mathcal{J}(R)$ and  $P \cap S=\emptyset$.
If $f(P)$ is a right $f(S)$-$\mathcal{J}$  of $R_2$, then $P$ is a right  $S$-$\mathcal{J}$ of $R_1$.
\end{theorem}

\begin{proof}
Suppose that $A_1B_1\subseteq P$ for ideals $A_1,B_1$ of $R$.
Then, $f(A_1)f(B_1)=f(A_1B_1)\subseteq f(P)$.
Since $f(P)$ is a right  $f(S)$-$\mathcal{J}$ of $R_2$, then, for some  $s\in S$,
\[
 \text{ either } f(B_1)\langle f(s)\rangle\subseteq f(P),  \text{ or } f(A_1)\langle f(s)\rangle\subseteq\mathcal{J}(R_2).
\]
  If  $f(B_1)\langle f(s)\rangle\subseteq f(P)$, then $B_1\langle s\rangle \subseteq f^{-1}(f(B_1\langle s\rangle)\subseteq f^{-1}(f(P))=P$, since $\text{\rm Ker}(f)\subseteq P$. 
  If $f(A_1)\langle f(s)\rangle\subseteq\mathcal{J}(R_2)$, in this case, we show that $A_1\langle s\rangle\subseteq\mathcal{J}(R_1)$ as the following.
  Let $J$ be an ideal of $R_1$ with $R_1/J \in \mathcal{S_\mathcal{J}}\cap\mathcal{K}$, then $\mathcal{J}(R_1)\subseteq J$, thus $\text{\rm Ker}(f)\subseteq J$. Hence, $f(R_1)/f(J)\cong (R_1/\text{\rm Ker}(f))/(J/\text{\rm Ker}(f)) \cong R_1/J$, so $R_2/f(J)=f(R_1)/f(J) \in \mathcal{S_\mathcal{J}}\cap\mathcal{K}$, and thus $f(A_1)\langle f(s)\rangle\subseteq\mathcal{J}(R_2)\subseteq f(J)$, thus  $A_1 \langle s\rangle\subseteq f^{-1}(f(A_1\langle s\rangle))\subseteq f^{-1}(f(J))=J$, for each $J$ in  $\cap\{I\lhd R: R/I\in \mathcal{S_\mathcal{J}}\cap\mathcal{K}\}$, and hence $A_1 \langle s\rangle\subseteq\mathcal{J}(R_1)$.    Thus, $P$ is the right  $S$-$\mathcal{J}$ of $R_1$.
\end{proof}

\begin{proposition}\label{RRPP} Let $R$ be a ring with identity, $P\lhd R$ and $(\mathcal{J}(R):\langle s\rangle)=\mathcal{J}(R)$ for some $s\in S$. The following are equivalent:
\end{proposition}
$(1)$ $P$ is a right $S$-$\mathcal{J}$-ideal of $R$ associated with $s\in S$.

\medskip
$(2)$  (i) $P\subseteq (\mathcal{J}(R):\langle s\rangle)$. 

\medskip
(ii) For all $x, y \in R$ with $xRy \subseteq P$, either $x\in(\mathcal{J}^*(P):\langle s\rangle)$ or $ y\in (P:\langle s\rangle)$.

\begin{proof} $(1)\Rightarrow(2)$ (i) Suppose  $P\not\subseteq (\mathcal{J}(R):\langle s\rangle)$, then, for all $x\in P{\setminus}(\mathcal{J}(R):\langle s\rangle)$ and any $y \in R$ we have $\langle x\rangle\langle y\rangle\subseteq P$, hence, by $(1)$, we obtain $ y\in (P:\langle s\rangle)$  for some $s\in S$, and thus, $R=(P:\langle s\rangle)$, which implies  $1\langle s\rangle\subseteq P$, a contradiction since $P\cap S=\emptyset$.  Thus, $P\subseteq  (\mathcal{J}(R):\langle s\rangle)$. 

\medskip
(ii) Now assume  $xRy \subseteq P$ for some $x$, $y$ of $R$. By $(3)$ of Proposition \ref{H}, either $x\in  (\mathcal{J}(R):\langle s\rangle)\subseteq(\mathcal{J}^*(P):\langle s\rangle)$ or $ y\in (P:\langle s\rangle)$ for some $s\in S$.

\medskip
$(2)\Rightarrow(1)$ Assume  $A_1A_2\subseteq P$  for some ideals $A_1$, $A_2$ of $R$. Assume that $A_1\langle s\rangle\not\subseteq \mathcal{J}(R)$ and $A_2\langle s\rangle\not\subseteq P$ for all $s\in S$. Then, there is $a_1\in A_1{\setminus}(\mathcal{J}(R):\langle s\rangle)$ and  $a_2 \in A_2{\setminus}(P:\langle s\rangle)$.  Hence, $a_1Ra_2\subseteq A_1A_2\subseteq P$, and by $(ii)$ we obtain either  $a_1\in(\mathcal{J}^*(P):\langle s\rangle)$ or $ a_2\in (P:\langle s\rangle)$.  Now since $P\subseteq (\mathcal{J}(R):\langle s\rangle)=\mathcal{J}(R)$, then, $\mathcal{J}^*(P)=\mathcal{J}(R)$. Thus, either $a_{1}\in  (\mathcal{J}(R):\langle s\rangle)$ or $ a_2\in (P:\langle s\rangle)$, a contradiction in both cases. Hence, $P$ is right $S$-$\mathcal{J}$ of $R$. 
\end{proof}

\begin{proposition}
Let $R$ has an identity and $P\lhd R$ such that $S\cap P=\emptyset$. 
\end{proposition}
$(1)$ If $P$ is a right  $S$-$\mathcal{J}$-ideal, then, $P\subseteq(\mathcal{J}(R):\langle s\rangle) $ for some $s\in S$.

\medskip
$(2)$ $\mathcal{J}(R)$ is a right $S$-$\mathcal{J}$ if and only if $\mathcal{J}(R)$ is a right $S$-prime.

\begin{proof}
$(1)$ Assume $P$ is a right $S$-$\mathcal{J}$-ideal. For all  $a\in P$, $\langle a\rangle\langle s_1\rangle \subseteq P$, for some $s_1\in S$. Hence, either $\langle a\rangle\langle s\rangle \subseteq \mathcal{J} (R)$ or $\langle s_1\rangle\langle s\rangle \subseteq P$ for some $s\in S$. If $\langle s_1\rangle\langle s\rangle \subseteq P$, then, for some $r\in R$, $s_2=s_1rs\in S$, hence, $\langle s_2\rangle\subseteq\langle s_1\rangle\langle s\rangle \subseteq P$, contradiction because  $S\cap P=\emptyset$. Hence, $\langle a\rangle\langle s\rangle \subseteq \mathcal{J} (R)$ for all $a\in P$, and thus, $P\subseteq(\mathcal{J}(R):\langle s\rangle) $.

\medskip
$(2)$ The proof is routine. 
\end{proof}

Recall that an ideal $P$ of a ring $R$ is called superfluous if 
$A + B = R$ for some ideal $B$ of $R$, then, $B = R$.

\begin{corollary}
For a  local ring $R$ with a unique maximal ideal $\mathcal{M}$, if $(\mathcal{J}(R):\langle s\rangle)$ is a $\mathcal{J}$-ideal  of $R$, then $I$ is a superfluous ideal for any  right $S$-$\mathcal{J}$-ideal $I$ of $R$.
\end{corollary}

\begin{proof}   Suppose $I+J=R$ for some ideal $J$ of $R$, then, there exist $i\in I$ and $j\in J$ such that $i+j=1$, hence, $1-j=i\in I$. Now since $I$ is   right $S$-$\mathcal{J}$ of $R$, then, by Proposition \ref{RRPP}, $I \subseteq(\mathcal{J}(R):\langle s\rangle)$, and since $(\mathcal{J}(R):\langle s\rangle)$ is a $\mathcal{J}$-ideal, then, by Proposition 1.5 of \cite{ref1},  $(\mathcal{J}(R):\langle s\rangle)\subseteq\mathcal{J}(R)\subseteq(\mathcal{J}(R):\langle s\rangle)$. Consequently,  
\[
1-j=i\in I\subseteq(\mathcal{J}(R):\langle s\rangle)=\mathcal{J}(R)=\mathcal{M}.
\]

Thus, $j$ is a unit, and hence, $J=R$, and $I$ is superfluous.
\end{proof}

It is interesting to explore the structure of rings where every ideal is an $S$-$\mathcal{J}$-ideal. As noted in \cite{Gulf}, Proposition 4.5 demonstrates that in the noncommutative Artinian local rings, every ideal is a $\mathcal{J}$-ideal and therefore also an $S$-$\mathcal{J}$-ideal. However, the question remains open regarding the structure of rings where every ideal is an $S$-$\mathcal{J}$-ideal but not necessarily a $\mathcal{J}$-ideal that is, where there exists at least one ideal that is not a $\mathcal{J}$-ideal.

\subsection{Conclusion}
In this paper, we have introduced and explored the concept of $S$-$\mathcal{J}$-ideals in both commutative and noncommutative rings.  We demonstrated that many properties of $\mathcal{J}$-ideals carry over to $S$-$\mathcal{J}$-ideals, thus offering a strong generalization. Our investigation into various ring constructions, including homomorphic image rings, quotient rings, cartesian product rings, polynomial rings, power series rings, idealization rings, and amalgamation rings, reveals the versatility and broad applicability of $S$-$\mathcal{J}$-ideals.

In noncommutative rings, we have defined right $S$-$\mathcal{J}$-ideals where $S$ is an $m$-system and demonstrated their equivalence with $S$-$\mathcal{J}$-ideals in the commutative case with identity. The provided examples illustrate the connections and distinctions between right $S$-prime ideals and $\mathcal{J}$-ideals, enriching the understanding of the new concept. 

This new class of ideals paves the way for further studies, such as  weakly $S$-$\mathcal{J}$, and quasi $S$-$\mathcal{J}$-ideals, in both commutative and noncommutative rings. These are closely related examples, and one can also, for example, explore $S$-$\mathcal{J}$-primary ideals in future research. 

We conclude the paper with several open questions that arise naturally from our investigation:

(1) Suppose the localization $IS^{-1}$ is a $\mathcal{J}$-ideal of $RS^{-1}$. Does it follow that $I$ is an $S$-$\mathcal{J}$-ideal of $R$?

(2) Is the set $S$ used in the localization the same as the one used in our generalization to define $S$-$\mathcal{J}$-ideals?

(3) How can we characterize rings where every ideal (disjoint from $S$) is an $S$-$\mathcal{J}$-ideal but not necessarily a $\mathcal{J}$-ideal?

These questions remain unanswered but open up avenues for further research.

\end{document}